\documentclass[a4paper, 12pt, reqno]{amsart}

\usepackage{lmodern}
\usepackage[english]{babel}
\usepackage[latin1]{inputenc}
\usepackage{microtype}
\usepackage{amsthm,amsfonts,amssymb,amsmath,mathrsfs}
\usepackage{color, graphicx}
\usepackage[all]{xy}
\usepackage[bookmarks=false]{hyperref} 
\usepackage[inline]{enumitem}
\usepackage{cleveref}
\usepackage{a4wide}

\makeatletter																																													
\DeclareMathSizes{\@xpt}{\@xpt}{6}{5}																																	
\makeatother																																													
\makeatletter
\def\namedlabel#1#2{\begingroup
    #2%
    \def\@currentlabel{#2}%
    \phantomsection\label{#1}\endgroup
}
\makeatother

\allowdisplaybreaks

\theoremstyle{plain}

\newtheorem{theorem}{Theorem}[section]
\newtheorem{lemma}[theorem]{Lemma}
\newtheorem{proposition}[theorem]{Proposition}
\newtheorem{corollary}[theorem]{Corollary}
\newtheorem*{theorem*}{Theorem}

\theoremstyle{definition}

\newtheorem{example}[theorem]{Example}

\theoremstyle{remark}
\newtheorem{remark}[theorem]{Remark}

\newcommand{\K}{\Bbbk}

\newcommand{\rmod}[1]{{#1}_{\bullet}^{}} 
\newcommand{\rhopfmod}[1]{{#1}_\bullet^\bullet}
\newcommand{\rcomod}[1]{{#1}_{}^{\bullet}}

\newcommand{\cl}[1]{\overline{#1}} 
\newcommand{\coinv}[2]{{#1}{}^{\mathrm{co}{#2}}} 
\newcommand{\lcoinv}[2]{{{}^{\mathrm{co}{#2}}{#1}}}
\newcommand{\inv}[2]{{\overline{#1}{}^{#2}}}
\newcommand{\linv}[2]{{{}^{#2}\overline{#1}}}

\newcommand{\rdual}[1]{{#1}^*} 
\newcommand{\tensor}[1]{\otimes_{{#1}}} 
\newcommand{\what}[1]{\widehat{#1}} 
\newcommand{\cmptens}[1]{\,\what{\otimes}_{#1}\,} 
\newcommand{\tildetens}[1]{\,\widetilde{\otimes}_{#1}\,}
\newcommand{\bartens}[1]{\,\overline{\otimes}_{#1}\,}

\renewcommand{\ker}{\mathsf{ker}} 

\newcommand{\id}{\mathsf{Id}} 

\newcommand{\Hom}[6]{{_{#1}^{#2}\mathsf{Hom}_{#3}^{#4}}\left({#5},{#6}\right)} 
\newcommand{\Homk}{\mathsf{Hom}} 
\newcommand{\End}[2]{\mathsf{End}_{#1}(#2)} 

\newcommand{\quasihopf}[1]{{}_{#1}^{\phantom{#1}}\M{}_{#1}^{#1}} 
\newcommand{\rhopf}[1]{\M_{#1}^{#1}}
\newcommand{\lhopf}[1]{{_{#1}^{#1}\M}}
\newcommand{\Rmod}[1]{\M^{}_{#1}} 
\newcommand{\Lmod}[1]{{}_{#1}^{}\M} 
\newcommand{\Rcomod}[1]{\M_{}^{#1}}

\newcommand{\cC}{{\mathcal C}}
\newcommand{\cD}{{\mathcal D}}

\newcommand{\cF}{{\mathcal F}}
\newcommand{\cG}{{\mathcal G}}

\newcommand{\cK}{{\mathcal K}}
\newcommand{\cL}{{\mathcal L}}

\newcommand{\cR}{{\mathcal R}}

\newcommand{\cU}{{\mathcal U}}

\newcommand{\M}{\mathfrak{M}} 

\newcommand{\ie}{i.e.~}
\newcommand{\eg}{e.g.~}

\title{Hopf modules, Frobenius functors and (one-sided) Hopf algebras}
\author{Paolo Saracco}
\address{D\'epartement de Math\'ematique, Universit\'e Libre de Bruxelles, Boulevard du Triomphe, \indent B-1050 Brussels, Belgium.}
\thanks{This paper was written while P. Saracco was member of the ``National Group for Algebraic and Geometric Structures and their Applications'' (GNSAGA-INdAM). He acknowledges FNRS support through a postdoctoral fellowship within the framework of the MIS Grant ``ANTIPODE'' (MIS F.4502.18, application number 31223212). He is grateful to Alessandro Ardizzoni and Joost Vercruysse for their willingness in discussing the content of the present paper and to the referees for their useful comments and suggestions. \\
\indent This version of the article has been accepted for publication, after peer review. The final publication is available at Elsevier via \href{https://doi.org/10.1016/j.jpaa.2020.106537}{doi.org/10.1016/j.jpaa.2020.106537}.}
\urladdr{sites.google.com/view/paolo-saracco}
\email{paolo.saracco@ulb.ac.be}

\usepackage{fancyhdr}

\begin{document}

\begin{abstract}
We investigate the property of being Frobenius for some functors strictly related with Hopf modules over a bialgebra and how this property reflects on the latter. In particular, we characterize one-sided Hopf algebras with anti-(co)multiplicative one-sided antipode as those for which the free Hopf module functor is Frobenius. As a by-product, this leads us to relate the property of being an FH-algebra (in the sense of Pareigis) for a given bialgebra with the property of being Frobenius for certain naturally associated functors.
\end{abstract}

\maketitle

\fancyhf{}
\renewcommand{\headrulewidth}{0pt}
\thispagestyle{fancy}
\cfoot{\smallskip\footnotesize $\begin{gathered}\includegraphics[scale=0.5]{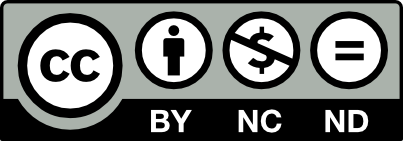}\end{gathered}$ \ \copyright\, 2020. This manuscript version is made available under the \href{https://creativecommons.org/licenses/by-nc-nd/4.0/}{CC-BY-NC-ND 4.0} license.}

\tableofcontents


\section*{Introduction}

An outstanding result of Morita \cite{Morita} claims that a $\K$-algebra extension $A\to B$ is Frobenius if and only if the restriction of scalars $U$ from $B$-modules to $A$-modules admits a two-sided adjoint, that is to say, if and only if $U$ is a \emph{Frobenius functor}. This established Frobenius functors as the categorical counterpart of Frobenius extensions, opening the way to the study of the Frobenius property in a broader sense (see \eg \cite{CaenepeelMilitaruZhu}).

An equally outstanding result of Pareigis \cite{Pareigis} claims that, under certain mild conditions, a $\K$-bialgebra $B$ is a finitely generated and projective Hopf algebra if and only if it is Frobenius as an algebra and the Frobenius homomorphism is a left integral on $B$. 

If we consider the free Hopf module functor $-\otimes B$ from the category of $\K$-modules to the category of (right) Hopf $B$-modules then it is well-known, under the name \emph{Structure Theorem of Hopf modules}, that $-\otimes B$ is an equivalence of categories if and only if $B$ admits an antipode. What seems to be not known is that this functor always fits into an adjoint triple $-\tensor{B}\K \, \dashv \, -\otimes B \,\dashv\, \coinv{(-)}{B}$ and the Structure Theorem describes when these are equivalences. It is natural then to ask ourselves what can be said if $-\otimes B$ is just a Frobenius functor. Surprisingly, the answer (see Theorem \ref{thm:sigmainv2}) involves the notion of \emph{one-sided Hopf algebras} introduced by Green, Nichols and Taft \cite{GreenNicholsTaft} and studied by Taft and collaborators \cite{IyerTaft, LauveTaft, NicholsTaft, RomoTaft1, RomoTaft2}: right Hopf algebras whose right antipode is an anti-bialgebra endomorphism are precisely those bialgebras for which $-\otimes B$ is Frobenius. A left-handed counterpart holds as well (Theorem \ref{thm:CharLeftHopf}) and merging the two together gives a new equivalent description of when a bialgebra is a Hopf algebra (Theorem \ref{thm:superEquiv}). 
An additional question which arises is how these achievements can be connected with Pareigis' classical result. In this direction we will show in Theorem \ref{thm:summingup} that, for a bialgebra $B$, being a Frobenius algebra whose Frobenius homomorphism is an integral in $B^*$ is strictly related to being Frobenius for certain functors naturally involved in the Structure Theorem. 

It is a well-known fact that there should exist a strong relationship between Hopf and Frobenius properties, as it can be deduced from many scattered results in the literature. Apart from Pareigis' work, let us mention that Larson and Sweedler \cite{LarsonSweedler} proved that the existence of an antipode for a finite-dimensional bialgebra $B$ over a PID is equivalent to the existence of a non-singular integral on $B$ and from this they deduced that finite-dimensional Hopf algebras over PID are always Frobenius. Hausser and Nill \cite{HausserNill} extended these results to quasi-Hopf algebras, Bulacu and Caenepeel \cite{BulacuCaenepeel} to dual quasi-Hopf algebras and Iovanov and Kadison \cite{IovanovKadison} addressed the question for the weak (quasi) Hopf algebra case.  Let us also recall the description of groupoids as special Frobenius objects in a suitable category given in \cite{Heunen}.
Following the spirit of these achievements, the results presented herein are intended to be a first step toward the investigation of the Frobenius-Hopf relationship by revealing connections between the property of being Hopf for bialgebras and the property of being Frobenius for certain functors. In a forthcoming paper \cite{Saracco-PreFrob}, we will develop this project by analysing, for example, the case of the functor $-\otimes B:\Lmod{B}\to \quasihopf{B}$.

Concretely, the paper is organized as follows. In Section \ref{sec:adjtriples} we recall some general facts about adjoint triples and Frobenius functors that will be needed later on. Section \ref{sec:HopfMod} is devoted to the study of when the Larson-Sweedler's free Hopf module functor is Frobenius. 
In Section \ref{sec:FrobeniusHopf} we will address some categorical implications of \cite{Pareigis} and we will investigate the connection between the Frobenius property for $-\otimes B:\M\to\rhopf{B}$, and for other strictly related functors, and the property of being an FH-algebra (or, equivalently, a Hopf algebra) for $B$.

\subsection*{Notations and conventions}

Throughout the paper, $\K$ denotes a commutative ring and $B$ a bialgebra over $\K$ with unit $u:\K\to B$, multiplication $m:B\otimes B\to B$, counit $\varepsilon:B\to \K$ and comultiplication $\Delta:B\to B\otimes B$. We write $B^+ = \ker(\varepsilon)$ for the augmentation ideal of $B$. The category of all (central) $\K$-modules is denoted by $\M$ and by $\M_B$, $\M^B$ and $\M^B_B$ (resp.~${_B\M}$, ${^B\M}$ and ${^B_B\M}$) we mean the categories of right (resp.~left) modules, comodules and Hopf modules over $B$, respectively. The unadorned tensor product $\otimes$ is the tensor product over $\K$ as well as the unadorned $\Homk$ stands for the space of $\K$-linear maps. The coaction of a comodule is denoted by $\delta$ and the action of a module by $\mu$, $\cdot$ or simply juxtaposition. In addition, if the context requires to report explicitly the (co)module structures on a $\K$-module $V$, then we use a full bullet, such as $V_\bullet$ or $V^\bullet$, to denote a given action or coaction respectively. By $V^u:=V\otimes \K^u$ and $V_\varepsilon:=V\otimes \K_\varepsilon$ we mean the trivial right comodule and right module structures on $V$ (analogously for the left ones).


\section{Preliminaries}\label{sec:adjtriples}

We recall some facts on adjoint triples and Frobenius functors that will be needed in the sequel.

\subsection{Adjoint triples}

For categories $\cC ,\cD$, an \emph{adjoint triple} is a triple of functors $\cL,\cR:\cC \rightarrow \cD $, $\cF:\cD \rightarrow \cC $ such that $\cL$ is left adjoint to $\cF$, which is left adjoint to $\cR$. It is called an \emph{ambidextrous adjunction} if $\cL\cong \cR$.
As a matter of notation, we set $\eta :\id \rightarrow \cF\cL,\epsilon :\cL\cF \rightarrow \id $ for the unit and counit of the left-most adjunction and $\gamma :\id \rightarrow \cR\cF,\theta :\cF\cR \rightarrow \id $ for the right-most one. 
If, in addition, $\cF$ is fully faithful, then we can consider the composition
\begin{equation*}
\sigma :=\left( \xymatrix{\cR \ar[r]^-{\left( \epsilon \cR\right) ^{-1}} & \cL\cF\cR \ar[r]^-{\cL\theta} & \cL}\right) .
\end{equation*}
Naturality of $\epsilon$ entails that $\sigma \cF\circ \gamma \circ \epsilon = \id$ and hence we have that 
\begin{equation}\label{eq:sigmaF}
\sigma \cF=\epsilon ^{-1}\circ \gamma ^{-1}
\end{equation}
is a natural isomorphism. Note also that 
\begin{equation}\label{eq:Fsigma}
\cF\sigma =\cF\cL\theta \circ \left( \cF\epsilon \cR\right) ^{-1}=\cF\cL\theta \circ \eta \cF\cR=\eta \circ \theta.
\end{equation}

\begin{remark}\label{rem:sigma}
One may consider $\left( \gamma \cL\right) ^{-1} \circ \cR\eta$ as well, but by resorting to the naturality of the morphisms involved, the invertibility of $\gamma$ and $\epsilon$ and the triangular identities, it turns out that $\gamma \cL\circ \cL\theta =\cR \eta \circ \epsilon \cR$.
\end{remark}

\begin{proposition}\label{prop:ambijunction}
If $\cF$ is fully faithful, then the adjoint triple $\cL\dashv \cF\dashv \cR$ is an ambidextrous adjunction if and only if $\sigma $ is a natural isomorphism.
\end{proposition}

\begin{proof}
If $\sigma $ is a natural isomorphism then $\cL\dashv \cF\dashv \cR$ is an ambidextrous adjunction. Conversely, if there exists a natural isomorphism $\tau :\cR\rightarrow \cL$ then the following diagram commutes 
\begin{equation*}
\xymatrix{
\cR \ar[r]^-{\cR\eta } \ar[d]_-{\tau} & \cR\cF\cL \ar[d]^-{\tau \cF\cL} \\ 
\cL \ar[r]_-{\cL\eta } & \cL\cF\cL
}
\end{equation*}
From the triangular identities of the adjunction $\cL\dashv \cF$ we have that $\cL\eta $ is a natural isomorphism.
Therefore, $\cR\eta $ is invertible and $\left( \cR\eta \right) ^{-1}\circ \gamma \cL$  gives an inverse for $\sigma $.
\end{proof}

\subsection{Frobenius pairs and functors}

A \emph{Frobenius pair} $(\cF,\cG)$ for $\cC $ and $\cD $ is a couple of functors $\cF:\cC \rightarrow \cD $ and $\cG:\cD \rightarrow \cC $ such that $\cG $\ is at the same time a left and a right adjoint to $\cF$. A functor $\cF$ is said to be \emph{Frobenius} if there exists a functor $\cG$ such that $(\cF,\cG)$ is a Frobenius pair.
Moreover, if $\cL\dashv \cF\dashv \cR$ is an adjoint triple, then it is an ambidextrous adjunction if and only if $(\cL,\cF)$ (equivalently, $(\cF,\cR)$) is a Frobenius pair. 

\begin{lemma}\label{lemma:FrobeniusSigma}
A fully faithful functor $\cF$ is Frobenius if and only if it is part of an adjoint triple $\cL\dashv \cF\dashv \cR$ where the canonical map $\sigma $ is a natural isomorphism.
\end{lemma}

\begin{proof}
If $\cF$ is Frobenius, then there exists a functor $\cG$ such that $(\cF,\cG)$ is a Frobenius pair. In particular, $\cG \dashv \cF \dashv \cG$ is an ambidextrous adjunction. By taking $\cL = \cG = \cR$ and by applying Proposition \ref{prop:ambijunction}, we have that $\cL \dashv \cF \dashv \cR$ is an adjoint triple with $\sigma$ a natural isomorphism. Conversely, if $\cF$ is part of an adjoint triple $\cL\dashv \cF\dashv \cR$ where the canonical map $\sigma $ is a natural isomorphism, then $\cL\dashv \cF\dashv \cR$ is an ambidextrous adjunction by Proposition \ref{prop:ambijunction} again and hence $\cF$ is Frobenius. 
\end{proof}

Since we are mainly interested in adjoint triples whose middle functor is fully faithful, Lemma \ref{lemma:FrobeniusSigma} allows us to study the Frobenius property by simply looking at the invertibility of the canonical map $\sigma$.

Recall finally from \cite{SepFunctors} that a functor $\cF:\cC\to\cD$ is said to be \emph{separable} if the natural transformation $\cF:\Hom{}{}{\cC}{}{\cdot}{\cdot} \to \Hom{}{}{\cD}{}{\cF(\cdot)}{\cF(\cdot)}$ splits. 
In light of Rafael's Theorem \cite[Theorem 1.2]{Rafael}, a left (resp.~right) adjoint is separable if and only if the unit (resp.~counit) of the adjunction is a split monomorphism (resp.~epimorphism).

\begin{proposition}\label{prop:equiv}
For a fully faithful functor $\cF:\cC \rightarrow \cD $, the following are equivalent
\begin{enumerate}[label=(\arabic*), ref=\emph{(\arabic*)}, leftmargin=1cm, labelsep=0.3cm]
\item\label{item:equiv1} $\cF$ is an equivalence;
\item\label{item:equiv2} $\cF$ admits a separable right adjoint $\cR$;
\item\label{item:equiv3} $\cF$ admits a separable left adjoint $\cL$.
\end{enumerate}
\end{proposition}

\begin{proof}
We prove that \ref{item:equiv3} implies \ref{item:equiv1} (the implication $\ref{item:equiv2}\Rightarrow \ref{item:equiv1}$ is analogous). If $\cR$ is separable then there exists a natural transformation $\tau:\id\to \cF\cR$ such that $\theta\circ\tau=\id$. Since the unit $\gamma$ is a natural isomorphism (because $\cF$ is fully faithful), the triangular identities imply that $\tau\cF = \cF\gamma$, as $\theta\cF\circ\tau\cF=\id$ and $\theta\cF=\cF\gamma^{-1}$. Now, naturality of $\tau$ implies that
\[
\tau\circ\theta = \cF\cR\theta\circ \tau\cF\cR = \cF\cR\theta\circ\cF\gamma\cR = \cF(\cR\theta\circ\gamma\cR) = \id.\qedhere
\]
\end{proof}

We conclude this section with the following result, for future reference. Recall from \cite[Definition 1.1]{Street} that a monad $(T:\cC\to\cC,\mu,u)$ is \emph{Frobenius} when it is equipped with a natural transformation $e:T\to\id_{\cC}$ such that there exists a natural transformation $\rho:\id_{\cC}\to T^2$ satisfying
\[
T\mu\circ \rho T = \mu T\circ T\rho \qquad \text{and} \qquad Te\circ \rho = u = eT\circ \rho.
\]
Moreover, recall that a functor $\cF:\cC\to\cD$ is said to be \emph{monadic} if it admits a left adjoint $\cL$ and the comparison functor $\cK:\cC\to\cD^{\mathbb{T}}$ is an equivalence of categories, where $\mathbb{T}=(\cF\cL,\eta,\cF\epsilon\cL)$ is the monad associated with the adjunction $\cL\dashv\cF$ and $\cD^{\mathbb{T}}$ is its Eilenberg-Moore category.

\begin{proposition}\label{prop:monadic}
Let $\cF:\cC\to\cD$ be monadic with left adjoint $\cL$. Then the monad $\mathbb{T}$ is Frobenius if and only if $\cF$ is Frobenius.
In particular, if $\cL\dashv \cF\dashv\cR$ is an adjoint triple with $\cF$ fully faithful, then $\mathbb{T}=(\cF\cL,\eta,\cF\epsilon\cL)$ is a Frobenius monad on $\cC$ if and only if $\cF$ is Frobenius.
\end{proposition}

\begin{proof}
By \cite[Proposition 1.5]{Street} we know that $\mathbb{T}$ is Frobenius if and only if the forgetful functor $\cU^{\mathbb{T}}:\cD^{\mathbb{T}}\to\cD$ is Frobenius. Denote by $\cF^{\mathbb{T}}:\cD\to\cD^{\mathbb{T}}$ the free algebra functor (which is left adjoint to $\cU^{\mathbb{T}}$) and recall that $\cF^{\mathbb{T}} = \cK\cL$ and that $\cU^{\mathbb{T}}\cK = \cF$. The following bijections
\begin{gather*}
\Hom{}{}{\cD}{}{\cF(X)}{Y} \cong \Hom{}{}{\cD}{}{\cU^{\mathbb{T}}\cK(X)}{Y}, \\
\Hom{}{}{\cC}{}{X}{\cL(Y)} \cong \Hom{}{}{\cD^{\mathbb{T}}}{}{\cK(X)}{\cK\cL(Y)} \cong \Hom{}{}{\cD^{\mathbb{T}}}{}{\cK(X)}{\cF^{\mathbb{T}}(Y)},
\end{gather*}
make it clear that $\cU^{\mathbb{T}}$ is left adjoint to $\cF^{\mathbb{T}}$ (\ie $\cU^{\mathbb{T}}$ is Frobenius) if and only if $\cF$ is left adjoint to $\cL$.
Concerning the second assertion of the statement, recall from \cite[Proposition 2.5]{ArdiMeniniPepe} that, in the stated hypotheses, the functor $\cF$ is monadic.
\end{proof}


\section{One-sided Hopf algebras and the free Hopf module functor}\label{sec:HopfMod}


In this section we study an example of an adjoint triple that naturally arises working with Hopf modules over a bialgebra. Deciding when this adjoint triple gives rise to a Frobenius functor leads to consider a certain weaker analogue of Hopf algebras, namely \emph{one-sided Hopf algebras}. 

Let $\left( B,m,u,\Delta ,\varepsilon \right) $ be a bialgebra over $\K$. 
It is well-known that $\left( \rmod{B},\Delta ,\varepsilon \right) $ is a comonoid in the monoidal category $\M _{B}$ of right $B$-modules and hence the forgetful functor $\cU ^{B}: \M_{B}^{B} \to \M_{B}$ is left adjoint to $-\otimes B : \M_{B} \to \M_{B}^{B}$. 
In addition, since ${_{\varepsilon }\K }$ is a $\left( B,\K \right) $-bimodule, the hom-tensor adjunction gives rise to another pair of adjoint functors between $\M$ and $\M_B$: $-\otimes_{B}{_\varepsilon\K} \ \dashv \ \Hom{}{}{}{}{{_\varepsilon\K}}{-}$.
Composing the two adjunctions, we get $\inv{\left( -\right)}{B} \ \dashv \ -\otimes B$ where, for every Hopf module $M$, the $\K$-module $\inv{M}{B}$ is the quotient 
\begin{equation*}
\inv{M}{B}=\frac{M}{MB^{+}} \cong \cU ^{B}\left( M\right) \otimes _{B}{_\varepsilon\K}
\end{equation*}
and, for every $V$ in $\M$, the Hopf module structure on $V\otimes B$ is given by
\begin{equation}\label{eq:TrivHopfMod} 
\delta \left( v\otimes b\right) = v\otimes b_{1}\otimes b_{2}, \qquad
\left( v\otimes b\right) \cdot a = v\otimes ba, 
\end{equation}
for every $v\in V,a,b\in B$ (the notation $b_{1}\otimes b_{2}$ stands for $\Delta(b)$, by resorting to Sweedler's Sigma Notation). On the other hand, since $(\rcomod{B},m,u) $ is also a monoid in the monoidal category $\M ^{B}$, we can join the two pairs of adjoint functors $-\otimes B \ \dashv \ \cU _{B}$ (between $\M_B$ and $\M_B^B$) and $-\otimes \K ^{u} \ \dashv \ \coinv{(-)}{B}$ (between $\M$ and $\M^B$) to get $-\otimes B \ \dashv \ \coinv{(-)}{B}$, where the Hopf module structure on $V\otimes B$ is the same of \eqref{eq:TrivHopfMod}, $\K ^{u}$ is the right $B$-comodule structure on $\K $ induced by $u:\K \rightarrow B$ and for every Hopf module $M$,
\begin{equation*}
M^{\mathrm{co}B}=\Hom{}{}{}{B}{\K ^{u}}{\cU _{B}\left(M\right) } =\left\{ m\in N\mid \delta \left( m\right) =m\otimes 1\right\} .
\end{equation*}
Summing up, we have an adjoint triple
\begin{equation}\label{eq:FundTriple}
\begin{gathered}
\xymatrix @R=20pt{
\M_{B}^{B} \ar@/_4ex/@<-0.3ex>[d]_-{\inv{(-)}{B}} \ar@/^4ex/@<+0.3ex>[d]^-{\coinv{(-)}{B}} \\ 
\M \ar[u]|-{-\otimes B}
}
\end{gathered}
\end{equation}
with units and counits given by
\begin{gather}\label{eq:unitscounits}
\begin{gathered}
\eta _{M} :M \rightarrow \inv{M}{B} \otimes B, \quad  m \mapsto \overline{m_{0}} \otimes m_{1}, \qquad \epsilon _{V} :\inv{V \otimes B}{B} \cong V, \quad  \overline{v \otimes b} \mapsto v\varepsilon \left( b\right) , \\
\gamma _{V} :V \cong \coinv{(V \otimes B)}{B}, \quad  v \mapsto v \otimes 1, \qquad \theta _{M} :\coinv{M}{B} \otimes B \rightarrow M, \quad  m \otimes b \mapsto m\cdot b,
\end{gathered}
\end{gather}
and we are exactly in the situation of \S\ref{sec:adjtriples}. The canonical morphism $\sigma $ is simply
\begin{equation*}
\sigma _{M}:M^{\mathrm{co}B}\rightarrow \inv{M}{B}, \quad m \mapsto  \overline{m},
\end{equation*}
and we want to investigate what can be said if this is a natural isomorphism, that is to say, we are interested in characterizing when the functor $-\otimes B$ is a Frobenius functor.

\begin{lemma}\label{lemma:sigmainv}
Given $M\in\rhopf{B}$, $\sigma_M$ is an isomorphism if and only if $M\cong \coinv{M}{B}\oplus MB^+$ as a $\K$-module.
\end{lemma}

\begin{proof}
Observe that $\sigma_M$ is injective if and only if $\coinv{M}{B}\cap MB^+ = 0$ and it is surjective if and only if $M=\coinv{M}{B}+MB^+$, whence it is bijective if and only if the canonical morphism $\coinv{M}{B}\oplus MB^+ \to M$ is an isomorphism.
\end{proof}

Henceforth, for the sake of simplicity, we will denote the Hopf module $B_{\bullet }\otimes B_{\bullet }^{\bullet }$ by $B\cmptens{}B$ or simply $\widehat{B}$. Explicitly, for all $a,b,c\in B$ its structures are given by
\begin{gather*}
\delta \left( a\otimes b\right) =a\otimes b_{1}\otimes b_{2}, \quad \left( a\otimes b\right) \cdot c =ac_{1}\otimes bc_{2}.
\end{gather*}

\begin{remark}
Notice that for $a_{i}\otimes b_{i}\in \coinv{ (B\cmptens{} B)}{B}$ we have 
\begin{equation}
a_{i}\otimes b_{i}=\left( a_{i}\varepsilon \left( b_{i}\right) \right) \otimes 1  \label{eq:coinvsimpl}
\end{equation}
because coinvariance implies that $a_{i}\otimes \left( b_{i}\right) _{1}\otimes \left( b_{i}\right)_{2}=a_{i}\otimes b_{i}\otimes 1$.
Thus every element in $\coinv{ (B\cmptens{} B)}{B}$ is of the form $x\otimes 1$ for $x\in B$.
\end{remark}

\begin{lemma}\label{lemma:sigmalinear}
The $\K$-modules $\coinv{ (B\cmptens{} B)}{B}$ and $\inv{B\cmptens{} B}{B}$ are left $B$-modules with actions
\begin{equation}
a\cdot \left( x\otimes 1\right) =ax\otimes 1\qquad \text{and}\qquad a\cdot  \overline{x_{i}\otimes y_{i}}=\overline{ax_{i}\otimes y_{i}},
\label{eq:leftaction}
\end{equation}
respectively. The canonical morphism $\sigma _{\widehat{B}}$ is left $B$-linear with respect to these actions.
\end{lemma}

\begin{proof}
Straightforward.
\end{proof}

\begin{proposition}\label{prop:sigmainv}
For a bialgebra $B$, $\sigma$ is a natural isomorphism if and only if there exists a $\K$-linear endomorphism $S:B\to B$ such that $S(1)=1$, $\varepsilon \circ S = \varepsilon$ and
\begin{gather}
a_1S(ba_2)=\varepsilon(a)S(b), \label{eq:super} \\
S(b)\otimes 1 = S(b_2)_1\otimes b_1S(b_2)_2, \label{eq:super2}
\end{gather}
 for all $a,b\in B$. In particular, if the foregoing conditions hold, then $\sigma_M^{-1}(\cl{m})=m_0\cdot S(m_1)$ for all $M\in \rhopf{B}$ and $m\in M$.
\end{proposition}

\begin{proof}
Observe that if we set $\sigma_M^{-1}(\cl{m}):=m_0\cdot S(m_1)$ for all $M\in \rhopf{B}$ and $m\in M$, then the following computations
\begin{gather*}
\sigma_M^{-1}(\cl{m\cdot b}) = m_0\cdot b_1S(m_1b_2) \stackrel{\eqref{eq:super}}{=} m_0\cdot S(m_1)\varepsilon(b), \\
\delta_M\left(\sigma_M^{-1}(\cl{m})\right) = m_0\cdot S(m_2)_1\otimes m_1S(m_2)_2 \stackrel{\eqref{eq:super2}}{=} m_0S(m_1)\otimes 1, \\
\sigma_M(\sigma_M^{-1}(\cl{m})) = \cl{m_0\cdot S(m_1)} = \cl{m_0}\varepsilon(S(m_1)) = \cl{m} \qquad \text{and} \\
\sigma_M^{-1}(\sigma_M(m')) = m'_0\cdot S(m_1') = m'\cdot S(1) = m',
\end{gather*}
for all $m\in M$, $m'\in\coinv{M}{B}$ and $b\in B$, imply that $\sigma_M^{-1}$ is well-defined and an inverse to $\sigma_M$. Thus we are left to prove the forward implication. Assume then that $\sigma$ is a natural isomorphism. Since $\delta_M:\rhopfmod{M}\to \rmod{M}\otimes \rhopfmod{B}$ is a morphism of Hopf modules, naturality of $\sigma^{-1}$ implies that
\[
\sigma_M^{-1} = (M\otimes \varepsilon)\circ \sigma_{M\otimes B}^{-1} \circ \cl{\delta_M}.
\]
In addition, since $f_m:\rmod{B}\to \rmod{M}, b\mapsto m\cdot b,$ is right $B$-linear for every $m\in M$, again naturality of $\sigma^{-1}$ implies that
\[
\sigma_{M\otimes B}^{-1}(\cl{m\otimes b}) = \sigma_{M\otimes B}^{-1}(\cl{f_m(1)\otimes b}) =\left(f_m\otimes B\right)\left(\sigma_{\what{B}}^{-1}(\cl{1\otimes b})\right)
\]
and hence
\begin{gather*}
\sigma_M^{-1}(\cl{m}) =  (M\otimes \varepsilon)\left( \sigma_{M\otimes B}^{-1}\left(\cl{m_0\otimes m_1}\right)\right) =  (M\otimes \varepsilon)\left(\left(f_{m_0}\otimes B\right)\left(\sigma_{\what{B}}^{-1}(\cl{1\otimes m_1})\right)\right) \\
= m_0\cdot (B\otimes \varepsilon)\sigma_{\what{B}}^{-1}(\cl{1\otimes m_1})
\end{gather*}
for all $m\in M$. Set $S(b) := (B\otimes \varepsilon)\sigma_{\what{B}}^{-1}(\cl{1\otimes b})$ for every $b\in B$, so that $\sigma_M^{-1}(\cl{m})=m_0\cdot S(m_1)$. Since $\sigma_{\what{B}}^{-1}(\cl{1\otimes b})\in\coinv{(B\cmptens{}B)}{B}$, 
\[
S(b)\otimes 1 \stackrel{\eqref{eq:coinvsimpl}}{=} \sigma_{\what{B}}^{-1}(\cl{1\otimes b}) {=} (1\otimes b)_0\cdot S\left((1\otimes b)_1\right) = S(b_2)_1\otimes b_1S(b_2)_2
\]
for all $b\in B$, which is \eqref{eq:super2}. Since 
\[
b\otimes 1 = \sigma_{\what{B}}^{-1}\left(\sigma_{\what{B}}(b\otimes 1)\right) = \sigma_{\what{B}}^{-1}\left(\cl{b\otimes 1}\right),
\]
we get that $1 = S(1)$ by considering $b=1$ and applying $B\otimes \varepsilon$ to both sides. Moreover, since $\sigma _{\widehat{B}}^{-1}$ is $B$-linear with respect to the actions of Lemma \ref{lemma:sigmalinear}, a direct computation shows that
\begin{align*}
a_1S(ba_2) & = a_{1}\left( B\otimes \varepsilon \right)\left( \sigma _{\widehat{B} }^{-1}\left( \overline{1\otimes ba_{2}}\right) \right) =\left( B\otimes \varepsilon \right) \left( \left( a_{1}\otimes 1\right) \left( \sigma _{\widehat{B}}^{-1}\left( \overline{1\otimes ba_{2}}\right) \right) \right) \\
&=\left( B\otimes \varepsilon \right) \left( a_{1}\cdot \left( \sigma _{\widehat{B}}^{-1}\left( \overline{1\otimes ba_{2}}\right) \right) \right) =\left( B\otimes \varepsilon \right) \left( \sigma _{\widehat{B}}^{-1}\left( \overline{a_{1}\otimes ba_{2}}\right) \right) \\
&=\left( B\otimes \varepsilon \right) \left( \sigma _{\widehat{B}}^{-1}\left( \overline{\left( 1\otimes b\right) \cdot a}\right) \right) =\varepsilon \left( a\right) \left( B\otimes \varepsilon \right) \left(\sigma _{\widehat{B}}^{-1}\left( \overline{\left( 1\otimes b\right) }\right)\right) =\varepsilon \left( a\right) S\left( b\right),
\end{align*}
which is \eqref{eq:super}. From these relations we can conclude also that
\[
\varepsilon \left( S\left( b\right) \right) =\varepsilon \left( b_{1}\right) \varepsilon \left( S\left( b_{2}\right) \right) =\varepsilon \left( b_{1}S\left( b_{2}\right) \right) =\varepsilon \left( \varepsilon \left( b\right)S(1) \right) =\varepsilon \left( b\right)
\]
for every $b\in B$ and this concludes the proof.
\end{proof}

\begin{proposition}\label{prop:boom}
If $\sigma_{\what{B}}$ is invertible, then the $\K$-linear endomorphism $S$ of $B$ given by $S(b):=\left(B\otimes \varepsilon\right)\left(\sigma_{\what{B}}^{-1}\left(\cl{1\otimes b}\right)\right)$ for all $b\in B$ is an anti-bialgebra morphism satisfying \eqref{eq:super}.
\end{proposition}

\begin{proof}
A closer inspection of the proof of Proposition \ref{prop:sigmainv} reveals that relations $\sigma_{\what{B}}^{-1}\left(\cl{1\otimes b}\right)=S(b)\otimes 1$, $S(1)=1$, \eqref{eq:super} and $\varepsilon\circ S=\varepsilon$ follow already from the invertibility of $\sigma_{\what{B}}$ alone. Moreover, the left $B$-linearity of $\sigma_{\what{B}}^{-1}$ with respect to the actions of Lemma \ref{lemma:sigmalinear} imply that
\begin{gather}
\sigma _{\widehat{B}}^{-1}\left( \overline{a\otimes b}\right) = a\cdot \sigma _{\widehat{B}}^{-1}\left( \overline{1\otimes b} \right) = aS\left( b\right) \otimes 1 \qquad\text{and so}  \label{eq:sigmainvsmallS} \\
\overline{aS\left( b\right) \otimes 1} = \sigma _{\widehat{B}}\left( aS\left( b\right) \otimes 1\right) \overset{\eqref{eq:sigmainvsmallS}}{=}\sigma _{\widehat{B}}\left( \sigma _{\widehat{B}}^{-1}\left( \overline{a\otimes b}\right) \right) =\overline{a\otimes b}  \label{eq:sigmainvsmallS2}
\end{gather}
for every $a,b\in B$. Now, consider the $\K$-module $\inv{B\cmptens{} B}{B}$ as endowed with the $B\otimes B$-module structure given by $\left( a\otimes b\right) \triangleright \left( \overline{x\otimes y}\right) = \overline{ax\otimes by}$ for $a,b,x,y\in B$. Then we have that
\begin{gather*}
S\left( b\right) S\left( a\right) \otimes 1 \overset{\eqref{eq:sigmainvsmallS}}{=}\sigma _{\widehat{B}}^{-1}\left( \overline{S\left( b\right) \otimes a} \right) = \sigma _{\widehat{B}}^{-1}\left( \left( 1\otimes a\right) \triangleright
\left( \overline{S\left( b\right) \otimes 1}\right) \right) \\
\overset{\eqref{eq:sigmainvsmallS2}}{=}\sigma _{\widehat{B}}^{-1}\left( \left( 1\otimes a\right) \triangleright \left( \overline{1\otimes b}\right) \right) = \sigma _{\widehat{B}}^{-1}\left( \overline{1\otimes ab}\right) \stackrel{\eqref{eq:sigmainvsmallS}}{=} S\left( ab\right) \otimes 1
\end{gather*}
and so $S\left( b\right) S\left( a\right) =S\left( ab\right) $. Concerning anti-comultiplicativity, consider the map $\lambda :B\otimes B\rightarrow B\otimes \inv{B\cmptens{} B}{B}$ given by 
\begin{equation*}
\lambda \left( a\otimes b\right):=  a_{1}S\left( b_{2}\right) \otimes \overline{a_{2}S\left( b_{1}\right) \otimes 1} \qquad \text{for all }a,b\in B
\end{equation*}
and pick $x_{i} z_{{i}_{1}}\otimes y_{i}z_{{i}_{2}}\in \left( B\otimes B\right) B^{+}$ (summation understood). We have that
\begin{gather*}
\lambda \left( x_{i}z_{{i}_{1}}\otimes y_{i}z_{{i}_{2}}\right) =\left( x_{i}z_{{i}_{1}}\right) _{1}S\left( \left( y_{i}z_{{i}_{2}}\right) _{2}\right) \otimes \overline{\left( x_{i}z_{{i}_{1}}\right) _{2}S\left( \left( y_{i}z_{{i}_{2}}\right) _{1}\right) \otimes 1} \\
=x_{{i}_{1}}z_{{i}_{1}}S\left( y_{{i}_{2}}z_{{i}_{4}}\right) \otimes \overline{x_{{i}_{2}}z_{{i}_{2}}S\left( y_{{i}_{1}} z_{{i}_{3}}\right) \otimes 1} \overset{\eqref{eq:super}}{=}x_{{i}_{1}}z_{{i}_{1}}S\left( y_{{i}_{2}}z_{{i}_{2}}\right) \otimes \overline{x_{{i}_{2}}S\left( y_{{i}_{1}}\right) \otimes 1} \\
\overset{\eqref{eq:super}}{=}\left( x_{{i}_{1}}S\left( y_{{i}_{2}}\right) \otimes \overline{x_{{i}_{2}}S\left( y_{{i}_{1}}\right) \otimes 1}\right) \varepsilon \left( z_{i}\right) =0
\end{gather*}
and so $\lambda $ factors through the quotient, giving a linear morphism $\inv{B\cmptens{} B}{B}\rightarrow B\otimes \inv{B\cmptens{} B}{B}$ that we denote by $\delta_{\mathsf{bar}} $. Set $\pi:B\cmptens{}B\to \inv{B\cmptens{}B}{B}$. From the following computation
\begin{equation*}
\delta_{\mathsf{bar}}  \left( \sigma _{\widehat{B}}\left( x\otimes 1\right) \right)=\delta_{\mathsf{bar}}  \left( \overline{x\otimes 1}\right) =x_{1}\otimes \overline{x_{2}\otimes 1}=\left( B\otimes \pi\right) \left( (\Delta\otimes B)\left( x\otimes 1\right) \right)
\end{equation*}
it follows that $\delta_{\mathsf{bar}}  =\left( B\otimes \pi\right)\circ (\Delta\otimes B) \circ \sigma _{\widehat{B}}^{-1}$. Thus
\[
S\left(b_2\right)\otimes\cl{S\left(b_1\right)\otimes 1} = \delta_{\mathsf{bar}}\left(\cl{1\otimes b}\right) = \left(\left( B\otimes \pi\right)\circ (\Delta\otimes B) \circ \sigma _{\widehat{B}}^{-1}\right)\left(\cl{1\otimes b}\right) = S(b)_1\otimes \cl{S(b)_2\otimes 1}
\]
for all $b\in B$. By applying $B\otimes \sigma_{\what{B}}^{-1}$ to both sides of this relation we conclude that
\begin{equation*}
S\left( b_{2}\right) \otimes S\left( b_{1}\right) =S\left( b\right) _{1}\otimes S\left( b\right) _{2}
\end{equation*}
for all $b\in B$.
\end{proof}

\begin{remark}
The interested reader may check that $\delta_{\mathsf{bar}}$ used in the foregoing proof make of $\inv{B\cmptens{} B}{B}$ a left $B$-comodule in such a way that $\sigma _{\widehat{B}}$ is left colinear, where on $\left( B\cmptens{} B\right) ^{\mathrm{co}B}$ one considers the coaction $\delta_{\mathsf{co}}\left( b\otimes 1\right) := b_{1}\otimes \left(b_{2}\otimes 1\right)$ for all $b\in B$.
\end{remark}

Recall that the module $\End{}{B}$ of $\K$-linear endomorphisms of a bialgebra $(B,m,u,\Delta, \varepsilon)$ is an algebra in a natural way: the unit is $u\circ \varepsilon$ and the multiplication is given by the \emph{convolution product} $f*g := m \circ (f\otimes g) \circ \Delta$.
A left (resp.~right) convolution inverse of the identity is called \emph{left} (resp.~\emph{right}) \emph{antipode} and if it exists then $B$ is called \emph{left} (resp.~\emph{right}) \emph{Hopf algebra} (see \cite{GreenNicholsTaft}). Summing up, the following is the first main result of the paper.

\begin{theorem}\label{thm:sigmainv2}
The following are equivalent for a bialgebra $B$
\begin{enumerate}[label=(\arabic*), ref=\emph{(\arabic*)}, leftmargin=1cm, labelsep=0.3cm]
\item\label{item:sigmainv2-1} $\sigma$ is a natural isomorphism;
\item\label{item:sigmainv2-1bis} for every $M\in\rhopf{B}$, $M\cong \coinv{M}{B}\oplus MB^+$ as a $\K$-module;
\item \label{item:sigmainv2-2} $\sigma_{\what{B}}$ is invertible;
\item\label{item:sigmainv2-3} $B$ is a right Hopf algebra with anti-(co)multiplicative right antipode $S$.
\end{enumerate}
In particular, if anyone of the above conditions holds, then $\sigma_M^{-1}(\cl{m})=m_0\cdot S(m_1)$ for all $M\in \rhopf{B}$ and $m\in M$.
\end{theorem}

\begin{proof}
The equivalence \ref{item:sigmainv2-1} $\Leftrightarrow$ \ref{item:sigmainv2-1bis} is the content of Lemma \ref{lemma:sigmainv}. Clearly, \ref{item:sigmainv2-1} implies \ref{item:sigmainv2-2}, which in turn implies \ref{item:sigmainv2-3} in light of Proposition \ref{prop:boom}. Moreover, if $B$ admits an anti-(co)multiplicative right antipode $S$, then the conditions \eqref{eq:super} and \eqref{eq:super2} are satisfied and hence \ref{item:sigmainv2-3} implies \ref{item:sigmainv2-1} by Proposition \ref{prop:sigmainv}. 
\end{proof}

\begin{remark}\label{rem:unitcounit}
Observe that the condition $\id _{B}\ast S=u\circ\varepsilon $ implies that $S\left( 1\right) =1S\left( 1\right)=u\varepsilon \left( 1\right) =1$ and that $\varepsilon \left( S\left(a\right) \right) =\varepsilon \left( a_{1}S\left( a_{2}\right) \right) =\varepsilon \left( u\varepsilon \left( a\right) \right) =\varepsilon \left(a\right) $ for every $a\in B$. Thus every right antipode is automatically unital and counital (analogously for left antipodes).
\end{remark}

\begin{example}[{\cite[Example 21]{GreenNicholsTaft}}]\label{ex:bello}
Let $\K$ be a field and consider the free algebra
\begin{equation*}
T:=\K\left\langle e_{ij}^{(k)} \mid 1\leq i,j\leq n, k\geq 0 \right\rangle
\end{equation*}
with bialgebra structure uniquely determined by
\begin{equation*}
\Delta \left( e_{ij}^{(k)}\right) = \sum_{h=1}^{n}e_{ih}^{(k)}\otimes e_{hj}^{(k)} \quad \text{and} \quad \varepsilon \left( e_{ij}^{(k)}\right) =\delta _{ij}
\end{equation*}
for all $1\leq i,j\leq n, k\geq 0$ and the assignment $s:= e_{ij}^{(k)} \mapsto e_{ji}^{(k+1)}$ for all $i,j,k$. The ideal $K\subseteq T$ generated by
\begin{equation*}
\left\{
\left. 
{\displaystyle \sum_{h=1}^{n} e_{hi}^{(k+1)}e_{hj}^{(k)} - \delta_{ij}1,} \quad {\displaystyle \sum_{h=1}^{n} e_{ih}^{(l)}e_{jh}^{(l+1)} - \delta_{ij}1} \ \right| \  k\geq 1,\ l\geq 0 ,\ 1\leq i,j\leq n
\right\}
\end{equation*}
is an $s$-stable bi-ideal and the bialgebra anti-endomorphism $S$ induced by $s$ is a right antipode but not a left one. Thus $(H,S)$ is a genuine right Hopf algebra satisfying the conditions of Theorem \ref{thm:sigmainv2}.
\end{example}

\begin{example}\label{ex:brutto}
In \cite[\S3]{NicholsTaft}, the authors exhibit a left Hopf algebra such that no left antipode is a bialgebra anti-endomorphism. Therefore, the requirements that the one-sided antipodes are either anti-comultiplicative or anti-multiplicative cannot be avoided, as in general a one-sided antipode can be neither.
\end{example}

Recall, from \cite{CaenepeelMilitaruZhu} for example, that a bimodule ${_SP_R}$ is a \emph{Frobenius bimodule} if and only if $(-\tensor{S}P,\Hom{}{}{R}{}{P}{-}):\M_S\to \M_R$ is a Frobenius pair of functors.
 
\begin{corollary}
Let $B$ be a bialgebra which is finitely generated and projective as a $\K$-module and denote by $B^c$ the co-opposite bialgebra (same algebra structure, co-opposite coalgebra structure). Then $B$ is a Hopf algebra if and only if ${_\K B_{B\#\rdual{B^c}}}$ is a Frobenius $(\K,B\#\rdual{B^c})$-bimodule.
\end{corollary}

\begin{proof}
Note that $B$ is a left $B^c$-comodule algebra, so that we can consider the category ${^{B^c}\M(B^c)_B}$ of Doi-Hopf modules over $(B^c,B)$ in the sense of \cite{Doi}. Note also that $\M^B_B=\M(B)_B^B$ and that we have an equivalence of categories $\M(B)_B^B\cong {^{B^c}\M(B^c)_B}$. Since $B$ is finitely generated and projective, we also have an equivalence ${^{B^c}\M(B^c)_B}\cong \M_{B\#\rdual{B^c}}$ (see \cite[Remark 1.3(b)]{Doi}). Thus, $\M_B^B\cong \M_{B\#\rdual{B^c}}$ and we can identify $-\otimes B:\M\to\M_B^B$ with $-\otimes {_\K B_{B\#\rdual{B^c}}}:\M_\K\to \M_{B\#\rdual{B^c}}$ and $\Hom{}{}{B}{B}{B}{-}$ with $\Hom{}{}{B\#\rdual{B^c}}{}{{_\K B_{B\#\rdual{B^c}}}}{-}$. In this context, and in light of \cite[Proposition 5]{GreenNicholsTaft}, Theorem \ref{thm:sigmainv2} can be restated by saying that $B$ is a (right) Hopf algebra if and only if $-\otimes B:\M\to \M_{B\#\rdual{B^c}}$ is Frobenius, if and only if ${_\K B_{B\#\rdual{B^c}}}$ is a Frobenius $(\K,B\#\rdual{B^c})$-bimodule.
\end{proof}

As we have seen with Example \ref{ex:bello}, there exist one-sided Hopf algebras whose one-sided antipode is a bialgebra anti-endomorphism. As a consequence, the right antipode of Theorem \ref{thm:sigmainv2} will not be a left convolution inverse in general. In light of this, let us proceed along a different path. The left-handed analogue of the previous construction holds, in the sense that we have another adjoint triple
\begin{equation*}
\linv{\left( -\right)}{B} \ \dashv  \  B\otimes -  \ \dashv  \  \lcoinv{\left( -\right)}{B}
\end{equation*}
between the category of $\K$-modules $\M $ and the category of left Hopf modules ${_{B}^{B}\M}$, where 
\begin{equation*}
{ \linv{M}{B}} =\frac{M}{B^{+}M}, \qquad \lcoinv{M}{B} =\left\{ m\in M\mid \delta \left( m\right) =1\otimes m\right\}
\end{equation*}
and $B\otimes V$ has the left module and comodule structures induced by those of $B$. We will denote with $\eta',\epsilon',\gamma'$ and $\theta'$ the units and counits of these adjunctions, analogously to \eqref{eq:unitscounits}.

We are again in the framework of \S\ref{sec:adjtriples}, the canonical morphism now being $\varsigma :\lcoinv{M}{B} \rightarrow { \linv{M}{B} }, \ m\mapsto \overline{m}$.
As before, we may also consider the component of $\varsigma $ corresponding to the Hopf module $\check{B} = B \check{\otimes} B :={ _{\bullet }^{\bullet }B} \otimes { _{\bullet }B}$, that is to say,
\begin{equation*}
\varsigma _{\check{B}}: \lcoinv{\left( B\check{\otimes} B\right)}{B} \rightarrow { \linv{B\check{\otimes} B}{B}} ; \quad 1\otimes b\mapsto \overline{1\otimes b},
\end{equation*}
and by mimicking the arguments used to prove Proposition \ref{prop:sigmainv} and Proposition \ref{prop:boom} one can prove the following result.

\begin{theorem}\label{thm:CharLeftHopf}
The following are equivalent for a bialgebra $B$.
\begin{enumerate}[label=(\arabic*), ref=\emph{(\arabic*)}, leftmargin=1cm, labelsep=0.3cm]
\item\label{item:CharLeftHopf3} $\varsigma $ is a natural isomorphism;
\item For every $M\in\lhopf{B}$, $M\cong \lcoinv{M}{B}\oplus B^+M$ as a $\K$-module;
\item\label{item:CharLeftHopf4} $\varsigma _{\check{B}}$ is invertible;
\item\label{item:CharLeftHopf1} $B$ is a left Hopf algebra with anti-(co)multiplicative left antipode $S'$.
\end{enumerate}
In particular, if anyone of the above equivalent conditions holds, then $\varsigma_M^{-1}(\cl{m}) = S'(m_{-1})m_0$ for all $M\in\lhopf{B}$ and $m\in M$.
\end{theorem}

In light of Theorem \ref{thm:sigmainv2} and Theorem \ref{thm:CharLeftHopf} we may now draw the following conclusions.

\begin{theorem}\label{thm:superEquiv}
The following are equivalent for a bialgebra $B$.
\begin{enumerate}[label=(\arabic*), ref=\emph{(\arabic*)}, leftmargin=1cm, labelsep=0.3cm]
\item\label{item:hopf1} $B$ is a Hopf algebra,
\item\label{item:hopf2} the canonical morphisms $\sigma $ and $\varsigma $ are natural isomorphisms,
\item\label{item:hopf3} the distinguished components $\sigma _{\widehat{B}} $ and $\varsigma _{\check{B}}$ are isomorphisms,
\item\label{item:hopf6} $\sigma _{\widehat{B}} $ is an isomorphism and either $\theta_{\widehat{B}}$ is surjective or $\eta_{\widehat{B}}$ is injective,
\item\label{item:hopf9} $\varsigma _{\check{B}} $ is an isomorphism and either $\theta'_{\check{B}}$ is surjective or $\eta'_{\check{B}}$ is injective,
\item\label{item:hopf4} either the functor $\coinv{(-)}{B}$ or the functor $\inv{(-)}{B}$ is separable,
\item\label{item:hopf5} either $\theta$ admits a natural section or $\eta$ admits a natural retraction,
\item\label{item:hopf7} either the functor $\lcoinv{(-)}{B}$ or the functor $\linv{(-)}{B}$ is separable,
\item\label{item:hopf8} either $\theta'$ admits a natural section or $\eta'$ admits a natural retraction.
\end{enumerate}
\end{theorem}

\begin{proof}
After recalling that $B$ is a Hopf algebra if and only if $-\otimes B$ is an equivalence, $\ref{item:hopf1}\Leftrightarrow\ref{item:hopf4}\Leftrightarrow\ref{item:hopf5}\Leftrightarrow\ref{item:hopf7}\Leftrightarrow\ref{item:hopf8}$ by Proposition \ref{prop:equiv} and the chain of implications \ref{item:hopf1} $\Rightarrow$ \ref{item:hopf2} $\Rightarrow$ \ref{item:hopf3} is clear. To go from \ref{item:hopf3} to \ref{item:hopf1} notice that Theorem \ref{thm:sigmainv2} and Theorem \ref{thm:CharLeftHopf} provide for us a right and a left convolution inverses of the identity morphism: $S$ and $S^{\prime }$ respectively. Since $\End{}{B}$ is a monoid with the convolution product, the two have to coincide and the resulting endomorphism $S^{\prime }=S$ is an antipode for $B$.
Thus, let us prove the equivalence between \ref{item:hopf1} and \ref{item:hopf6}. The only non trivial implication is $\ref{item:hopf6} \Rightarrow \ref{item:hopf1}$, whence assume that $\sigma _{\widehat{B}} $ is an isomorphism and that $\eta_{\widehat{B}}:B\cmptens{} B\to \inv{B\cmptens{} B}{B}\otimes B, x_i\otimes y_i\mapsto \overline{x_i\otimes y_{i_1}}\otimes y_{i_2},$ is injective (the proof with $\theta_{\widehat{B}}$ surjective is analogous\footnote{Note that $\theta_{\widehat{B}}$ is the Hopf-Galois map $\beta:B\otimes B\to B\otimes B: a\otimes b\mapsto ab_1\otimes b_2$.}). In light of relation \eqref{eq:Fsigma} we deduce immediately that $\eta_{\widehat{B}}$ has to be surjective as well and hence an isomorphism of Hopf modules, which is also $B$-linear with respect to the left actions
\begin{equation*}
b\triangleright \left(x_i\otimes y_i\right) = bx_i\otimes y_i \quad \text{and} \quad b\triangleright \left(\overline{x_i\otimes y_i}\otimes z_i\right) = \overline{bx_i\otimes y_i}\otimes z_i
\end{equation*}
on $B\cmptens{} B$ and $\inv{B\cmptens{} B}{B}\otimes B$ respectively. For all $b\in B$, set $\nu(b):=(B\otimes \varepsilon)\left(\eta_{\widehat{B}}^{-1}\left(\overline{1\otimes b}\otimes 1\right)\right)$. This gives an endomorphism $\nu$ of $B$. Since $\eta_{\widehat{B}}^{-1}$ is $B$-bilinear and $B$-colinear, we have that
\begin{align*}
\eta_{\widehat{B}}^{-1}\left(\overline{1\otimes b}\otimes 1\right) & = (B\otimes \varepsilon\otimes B)\left((B\otimes \Delta)\left(\eta_{\widehat{B}}^{-1}\left(\overline{1\otimes b}\otimes 1\right)\right)\right) \\
 & = (B\otimes \varepsilon\otimes B)\left(\eta_{\widehat{B}}^{-1}\left(\overline{1\otimes b}\otimes 1\right)\otimes 1\right) = \nu(b)\otimes 1
\end{align*}
and hence $\eta_{\widehat{B}}^{-1}\left(\overline{x_i\otimes y_i}\otimes z_i\right) = x_i\nu(y_i)z_{i_1}\otimes z_{i_2}$ for every $x_i\otimes y_i\otimes z_i\in B\otimes B\otimes B$. Now, for every $b\in B$ we have
$$1\otimes b = \eta_{\widehat{B}}^{-1}\left(\eta_{\widehat{B}}\left(1\otimes b\right)\right) = \nu(b_1)b_2\otimes b_3$$
and so, by applying $B\otimes \varepsilon$ to both sides, $1\varepsilon(b)=\nu(b_1)b_2$, \ie $\nu$ is a left convolution inverse of the identity. Since we already have a right one, the two have to coincide, giving an antipode for $B$.
The proof of the equivalence between \ref{item:hopf1} and \ref{item:hopf9} is similar.
\end{proof}

\begin{remark}
\begin{enumerate}[label=(\arabic*), ref={(\arabic*)}, leftmargin=1cm, labelsep=0.3cm]
\item\label{item:reph1} By rephrasing \ref{item:hopf6} of Theorem \ref{thm:superEquiv} in functorial terms we have that a bialgebra $B$ is a Hopf algebra if and only if $-\otimes B$ is Frobenius and either $\inv{(-)}{B}$ is faithful or $\coinv{(-)}{B}$ is faithful (and analogously on the other side). 
\item With Theorem \ref{thm:superEquiv}, we implicitly proved another structure theorem for Hopf modules: a bialgebra $B$ is a Hopf algebra if and only if the morphism $\eta_{\widehat{B}}$ is an isomorphism, if and only if every Hopf module $M$ over $B$ satisfies $M\cong \inv{M}{B}\otimes B$. Indeed, if $\eta_{\widehat{B}}$ is invertible then Remark \ref{rem:sigma} entails that $\sigma_{\widehat{B}}$ is invertible as well and hence we conclude by the argument in \ref{item:reph1}. This is the coassociative analogue of the structure theorem for quasi-Hopf bimodules \cite[Theorem 4]{Saracco}.
\end{enumerate}
\end{remark}

Let us conclude this section with the following interesting remark, linking the theory we developed here with Hopfish algebras.

\begin{remark}[Hopfish algebras]\label{rem:Hopfish}
Consider a bialgebra $\left( B,m,u,\Delta ,\varepsilon \right) $ and its
modulation
\begin{equation*}
\boldsymbol{\epsilon } := {{}_{\K }\K {}_{\varepsilon}} \qquad \boldsymbol{\Delta } := {{}_{B\otimes B}\left( B\otimes B\right){} _{\Delta }}
\end{equation*}
in the sense of \cite{Hopfish}. If we consider the Hopf module $B_{\bullet }\otimes B_{\bullet }^{\bullet }$, then
\begin{equation*}
\inv{B\cmptens{} B}{B}=\frac{\cU ^{B}\left( B_{\bullet }\otimes B_{\bullet }^{\bullet }\right) }{\cU ^{B}\left( B_{\bullet }\otimes
B_{\bullet }^{\bullet }\right) B^{+}}=\frac{B\otimes B}{\left( B\otimes B\right) \Delta \left( \ker\left( \varepsilon \right) \right) }=:\boldsymbol{S}.
\end{equation*}
By construction, $\boldsymbol{S}$ is just a $\K$-module, but we may endow it with the $B\otimes B$-module structure induced by the left multiplication, that is, $\left( a\otimes b\right) \triangleright \left( \overline{x\otimes y}\right) :=\overline{ax\otimes by}$. Recall from \cite[Theorem 4.2]{Hopfish} that $\boldsymbol{S}$ is a preantipode for the modulation of $B$. If we assume in addition that the distinguished morphism $\sigma _{\widehat{B}}$ is invertible, then the left $B$-linear morphism
\begin{equation*}
\xymatrix @C=25pt @R=0pt {
B \ar[r]^-{\cong} & \coinv{\left( B\cmptens{} B\right) }{B} \ar[r]^-{\sigma_{\widehat{B}}} & \inv{B\cmptens{} B}{B} = {\displaystyle \frac{B\otimes B}{\left(B\otimes B\right) \Delta \left( \ker\left( \varepsilon \right)\right) }} \\
b \ar@{|->}[rr] & & \left( b\otimes 1\right) +\left( B\otimes B\right) \Delta \left( \ker\left( \varepsilon \right) \right)
}
\end{equation*}
is invertible and hence $\boldsymbol{S}$ is a free left $B$-module of rank one generated by the class of $1\otimes 1$. Summing up, if $\sigma _{\widehat{B}}$ is invertible then $\left( \boldsymbol{B}, \boldsymbol{\Delta },\boldsymbol{\epsilon },\boldsymbol{S}\right) $ is a Hopfish algebra.

An interesting question which remains open is if the converse is true as well, that is to say, if we can characterize Hopfish algebras which are modulations of bialgebras in terms of the invertibility of $\sigma$.
\end{remark}


\section{Adjoint pairs and triples related to Hopf modules and FH-algebras}\label{sec:FrobeniusHopf}

As we have seen at the beginning of \S\ref{sec:HopfMod}, the adjoint triple \eqref{eq:FundTriple} studied in the previous section is just one member of a family of adjunctions appearing in the study of Hopf modules. In the present section we will spend a few words concerning the others and the property of being Frobenius for them and we will address the question concerning the relationship with Pareigis' results \cite{Pareigis}.

Let $M\in\rhopf{B}$. For the sake of clearness, for every $N\in\Rmod{B}$ we set $N\cmptens{}M:=\rmod{N}\otimes \rhopfmod{M}$, for every $P\in\Rcomod{B}$ we set $P\tildetens{}M:=\rcomod{P}\otimes \rhopfmod{M}$ and for every $V\in\M$ we set $V\bartens{}M:=V\otimes \rhopfmod{M}$ in $\rhopf{B}$. The notation $-\otimes M$ will be reserved for the functor $\M\to\Rmod{B}$, $V \mapsto V\otimes \rmod{M}$. 

Given a bialgebra $B$ over a commutative ring $\K$, it is straightforward to check that the category of left $B$-modules is not only monoidal, but in fact a (right) closed monoidal category with internal hom-functor $\Hom{}{}{B}{}{B\otimes N}{-}$ for all $N\in \Rmod{B}$. 

\begin{lemma}[compare with {\cite[\S2.1]{PhD}}, {\cite[Proposition 3.3]{Schauenburg-dualsdoubles}}]\label{lemma:modclosed}
Let $B$ be a bialgebra. The category ${\M_B}$ of right $B$-modules is left and right-closed. Namely, we have bijections
\begin{gather}
\xymatrix{
\Hom{}{}{B}{}{M\otimes N}{P} \ar@<+0.5ex>[r]^-{\varphi} & \Hom{}{}{B}{}{M}{\Hom{}{}{B}{}{B\otimes N}{P}}  \ar@<+0.5ex>[l]^-{\psi},
} \label{eq:lmodclosed}\\
\xymatrix{
\Hom{}{}{B}{}{N\otimes M}{P} \ar@<+0.5ex>[r]^-{\varphi'} & \Hom{}{}{B}{}{M}{\Hom{}{}{B}{}{N\otimes B}{P}}  \ar@<+0.5ex>[l]^-{\psi'},
}\notag
\end{gather}
natural in $M$ and $P$, given explicitly by
\begin{gather*}
\varphi(f)(m):a\otimes n \mapsto f(m\cdot a\otimes n), \qquad \psi(g):m\otimes n \mapsto g(m)(1\otimes n),\\
\varphi'(f)(m):n\otimes a \mapsto f(n\otimes m\cdot a), \qquad \psi'(g):n\otimes m \mapsto g(m)(n\otimes 1), 
\end{gather*}
where the right $B$-module structures on $\Hom{}{}{B}{}{B\otimes N}{P}$ and $\Hom{}{}{B}{}{N\otimes B}{P}$ are induced by the left $B$-module structure on $B$ itself.
\end{lemma}

\begin{lemma}
For $N\in\rhopf{B}$, the natural bijection \eqref{eq:lmodclosed} induces a bijection
\begin{equation*}
\Hom{}{}{B}{B}{M\cmptens{} N}{P}\cong \Hom{}{}{B}{}{M}{\Hom{}{}{B}{B}{B\cmptens{} N}{P}}
\end{equation*}
natural in $M\in\M_B$ and $P\in\M^B_B$. Thus, the functor $\Hom{}{}{B}{B}{B\cmptens{} N}{-}:\M_B^B\to\M_B$ is right adjoint to the functor $-\cmptens{}  N:\M_B\to\M_B^B$
\end{lemma}

\begin{proof}
We refer to the notation used in Lemma \ref{lemma:modclosed}. We already know that for every $f\in \Hom{}{}{B}{B}{M\cmptens{} N}{P}$, $\varphi(f)\in \Hom{}{}{B}{}{\rmod{M}}{\Hom{}{}{B}{}{\rmod{B}\otimes \rmod{N}}{\rmod{P}}}$. 
In addition, 
\[
\big(\varphi(f)(m)\big)(a\otimes n_0)\otimes n_1 = f(m\cdot a\otimes n_0)\otimes n_1 = f(m\cdot a\otimes n)_0\otimes f(m\cdot a\otimes n)_1
\]
for all $m\in M, n\in N, a\in B$, whence it is also colinear. For what concerns $\psi$, we know that $\psi(g)\in \Hom{}{}{B}{B}{\rmod{M}\otimes \rmod{N}}{\rmod{P}}$ for all $g\in \Hom{}{}{B}{}{M}{\Hom{}{}{B}{B}{B\cmptens{} N}{P}}$. 
In addition,
\[
\psi(g)(m\otimes n_0)\otimes n_1 = g(m)(1\otimes n_0)\otimes n_1 = g(m)(1\otimes n)_0\otimes g(m)(1\otimes n)_1
\]
for every $m\in M$, $n\in N$. Naturality is left to the reader.
\end{proof}

Summing up, we can consider the following family of adjunctions strictly connected with Hopf $B$-modules and the Structure Theorem:

\begin{gather}\label{eq:AllAdj}
\begin{gathered}
\xymatrix{
\rhopf{B} \ar@/^/[d]^-{U_B} \\
\Rcomod{B} \ar@/^/[u]^-{-\tildetens{}B}
} \qquad
\xymatrix{
\Rcomod{B} \ar@/^/[d]^-{\coinv{(-)}{B}} \\
\M \ar@/^/[u]^-{-\otimes \K^u}
} \qquad
\xymatrix{
\rhopf{B} \ar@/_/@<-2ex>[d]_-{\cl{(-)}^B} \ar@/^/@<+2ex>[d]^-{\coinv{(-)}{B}} \\
\M \ar[u]|-{-\bartens{}B}
} \qquad
\xymatrix{
\Rcomod{B} \ar@/_/@<-2ex>[d]_-{U'} \ar@/^/@<+2ex>[d]^-{\Hom{}{}{}{B}{B}{-}} \\
\M \ar[u]|-{-\otimes'B}
} \\
\xymatrix{
\Rmod{B} \ar@/_/[d]_-{-\tensor{B}\K} \\
\M \ar@/_/[u]_-{-\otimes \K_{\varepsilon}}
}
\qquad
\xymatrix{
\Rmod{B} \ar[d]|-{U} \\
\M \ar@/^/@<+1ex>[u]^-{-\otimes B} \ar@/_/@<-1ex>[u]_-{\Hom{}{}{}{}{B}{-}}
} \qquad
\xymatrix{
\rhopf{B} \ar@/_/@<-2ex>[d]_-{U^B} \ar@/^/@<+2ex>[d]^-{\Hom{}{}{B}{B}{B\cmptens{}B}{-}} \\
\Rmod{B} \ar[u]|-{-\cmptens{}B}
}
\end{gathered}
\end{gather}
For every $V\in \M$, $N\in\Rmod{B}$ and $P\in\Rcomod{B}$ we have
\begin{gather}\label{eq:compositions}
\begin{gathered}
U^B(V\bartens{} B) = V\otimes B, \qquad \coinv{(N\cmptens{}B)}{B} \cong U(N), \\
U_B\left(V\bartens{} B\right) = V\otimes'B \quad \text{and} \quad \cl{P\tildetens{}B}^B \cong U'(P).
\end{gathered}
\end{gather}

\begin{proposition}\label{prop:ManyAdj}
The following assertions hold.
\begin{enumerate}[label=(\arabic*), ref=\emph{(\arabic*)}, leftmargin=1cm, labelsep=0.3cm]
\item\label{item:ManyAdj1} If $-\tildetens{}B$ is Frobenius and $\Hom{}{}{}{B}{U_B(M)}{V^u}\cong\Hom{}{}{}{}{\coinv{M}{B}}{V}$ naturally in $M\in\rhopf{B},V\in\M$ then $-\bartens{}B$ is Frobenius.
\item\label{item:ManyAdj2} If $-\cmptens{}B$ is Frobenius and $\Hom{}{}{B}{}{V_\varepsilon}{U^B(M)}\cong\Hom{}{}{}{}{V}{\cl{M}^{B}}$ naturally in $M\in\rhopf{B},V\in\M$ then $-\bartens{}B$ is Frobenius.
\item\label{item:ManyAdj3} If $-\bartens{}B$ and $-\cmptens{}B$ are Frobenius then $-\otimes B$ is Frobenius.
\item\label{item:ManyAdj4} If $-\bartens{}B$ and $-\tildetens{}B$ are Frobenius then $-\otimes' B$ is Frobenius.
\item\label{item:ManyAdj5} If $-\tildetens{}B$ is Frobenius and $\Hom{}{}{}{B}{U_B(M)}{V^u}\cong\Hom{}{}{}{}{\coinv{M}{B}}{V}$ naturally in $M\in\rhopf{B},V\in\M$ then $-\otimes' B$ is Frobenius.
\item\label{item:ManyAdj6}  If $-\cmptens{}B$ is Frobenius and $\Hom{}{}{B}{}{V_\varepsilon}{U^B(M)}\cong\Hom{}{}{}{}{V}{\cl{M}^{B}}$ naturally in $M\in\rhopf{B},V\in\M$ then $-\otimes B$ is Frobenius.
\end{enumerate}
\end{proposition}

\begin{proof}
Straightforward.
\end{proof}

\begin{remark}\label{rem:CMZ}
Concerning conditions under which the adjunctions \eqref{eq:AllAdj} give rise to Frobenius pairs, the interested reader may check \cite[\S3.3 and \S3.4]{CaenepeelMilitaruZhu}.
\end{remark}


In \cite{Pareigis}, Pareigis proved that for a bialgebra $B$ over a commutative ring $\K$ the following assertions are equivalent: \begin{enumerate*} \item $B$ is a Hopf algebra, finitely generated and projective as a $\K$-module, such that $\int_l\rdual{B}\cong \K$ and \item $B$ is Frobenius as an algebra and its Frobenius homomorphism is a left integral on $B$\end{enumerate*} (see also \cite{KadisonStolin}). We conclude this section by discussing some categorical implications of this result.

\begin{remark}
For the sake of honesty and correctness, let us point out that in \cite{Pareigis} there is no explicit reference to the fact that if $B$ is Frobenius as an algebra and its Frobenius homomorphism $\psi$ is a left integral on $B$ then $\int_l\rdual{B}\cong \K$. Nevertheless, the subsequent argument is a straightforward consequence of the results therein. If a bialgebra $B$ is Frobenius as an algebra and $\psi\in\int_l\rdual{B}$, then by \cite[Theorem 2]{Pareigis} $B$ is a Hopf algebra and it is finitely generated and projective over $\K$. By \cite[Theorem 3]{Pareigis}, $\int_lB\cong \K$. As observed at \cite[page 596]{Pareigis}, $\rdual{B}$ is a finitely generated and projective Hopf algebra as well, $B^{**}\cong B$ as Hopf algebras and $\int_lB\cong \int_lB^{**}$. Therefore, by \cite[Theorem 1]{Pareigis}, $B^*$ is a Frobenius algebra. By \cite[Theorem 3]{Pareigis} again, $\int_lB^*\cong \K$.
\end{remark}

Let us begin by recalling some facts about Frobenius algebras. A $\K$-algebra $A$ is Frobenius if it is finitely-generated and projective as a $\K$-module and $A\cong \rdual{A}$ as right (or left) $A$-modules. This is equivalent to say that there exist an element $e:=e^1\otimes e^2\in A\otimes A$ (summation understood) and a linear map $\psi:A\to\K$ such that $ae=ea$ for all $a\in A$ and
\begin{equation}\label{eq:FrobRel}
e^1\psi(e^2) = 1 = \psi(e^1)e^2.
\end{equation}
The element $e$ is called a \emph{Casimir element} and the morphism $\psi$ a \emph{Frobenius homomorphism}. The pair $(\psi,e)$ is a \emph{Frobenius system} for $B$. A Frobenius homomorphism $\psi$ is a free generator of $\rdual{A}$ as a right (resp. left) $A$-module and the isomorphism $\Psi:A\to \rdual{A}$ is given by $\Psi(a)=\psi\cdot a$ (resp. $\Psi(a)=a\cdot \psi$), where $(\psi\cdot a)(b)=\psi(ab)$ for every $a,b\in B$. If $A$ is Frobenius and it is also augmented with augmentation $\varepsilon : A \to \K$, then there exists $T\in A$ such that $\varepsilon = \Psi(T) = \psi\cdot T$. It is called a \emph{right norm} in $A$ with respect to $\psi$. In particular, $\psi(T)=1$. If $e$ is a Casimir element such that \eqref{eq:FrobRel} holds (we will naively call it the Casimir element corresponding to $\psi$), then $T=\varepsilon(e^1)e^2$, because
\[
\psi(\varepsilon(e^1)e^2a) = \varepsilon(ae^1)\psi(e^2) =  \varepsilon(ae^1\psi(e^2)) =\varepsilon(a)
\]
for every $a\in A$ and $\Psi$ is invertible. In particular, $T$ is a right integral in $A$. Analogously, one may call \emph{left norm} an element $t\in A$ such that $t\cdot \psi=\varepsilon$ and in this case the identity $t=e^1\varepsilon(e^2)$ tells us that $t$ is a left integral. Finally, if $B$ is a bialgebra which is also a Frobenius algebra such that the Frobenius morphism $\psi$ is a right integral in $\rdual{B}$, then we call $B$ an \emph{FH-algebra} by mimicking \cite{KadisonStolin,Pareigis-cohomology}. 

\begin{remark}
If we consider a right-handed analogue of Pareigis' results, then we have that any finitely generated and projective Hopf algebra $B$ with $\int_r\rdual{B}\cong \K$ is Frobenius with Frobenius morphism $\psi\in\int_r\rdual{B}$ (by using the Structure Theorem for left Hopf modules). Conversely, if $B$ is an FH-algebra with Frobenius morphism $\psi$ and if $T$ is a right norm in $B$ with respect to $\psi$, then $B$ is a (finitely generated and projective) Hopf algebra, where the antipode is given by $S(a)=\psi(T_1a)T_2$ for all $a\in B$ (see \cite{KadisonStolin}). Moreover, if $t$ is a left norm in $B$ with respect to $\psi$, then the assignment $B\to B:a\mapsto \psi(at_1)t_2$ provides an inverse $S^{-1}$ for the antipode $S$, in light of \cite[Proposition 10.5.2(a)]{Radford} for example. We will often make use of these facts in what follows, as well as of the fact that for a finitely generated and projective $\K$-bialgebra $B$, $\psi \in \int_rB^*$ if and only if $\psi(a_1)a_2=\psi(a)1$ for all $a\in B$.
\end{remark}

\begin{lemma}\label{lemma:CoinvCasimir}
If $B$ is an FH-algebra, then the Casimir element $e$ corresponding to $\psi\in \int_r\rdual{B}$ satisfies $e^1_1\otimes e^2_1\otimes e^1_2e^2_2 = e^1\otimes e^2\otimes 1$, that is to say, $e\in\coinv{\left(B^\bullet\otimes B^\bullet\right)}{B}$.
\end{lemma}

\begin{proof}
Let $T$ be a right norm in $B$ with respect to $\psi$ and $t$ be a left norm instead. In light of \cite[Proposition 4.2]{KadisonStolin}, the element $S^{-1}(T_2)\otimes T_1 = \psi(T_2t_1)t_2\otimes T_1$ is the Casimir element corresponding to $\psi$ and it is easy to see that it is coinvariant with respect to the coaction of the statement.
\end{proof}

Henceforth, let us assume that $B$ is an FH-algebra with Frobenius morphism $\psi\in\int_r\rdual{B}$ and Casimir element $e=e^1\otimes e^2$.

\begin{proposition}\label{prop:ForgetModFrobenius}
The assignment
\[
\Hom{}{}{}{B}{U_B(M)}{P} \to \Hom{}{}{B}{B}{M}{P\tildetens{} B} : \quad f \mapsto \left[ m \mapsto f(me^1)\otimes e^2\right]
\]
is a bijection, natural in $M\in\M^B_B$ and $P\in\M^B$, with inverse
\[
\Hom{}{}{B}{B}{M}{P\tildetens{} B} \to \Hom{}{}{}{B}{U_B(M)}{P}: \quad g \mapsto \left[ m \mapsto (P\otimes \psi)(g(m)) \right].
\]
In particular, the functor $U_B:\M^B_B\to \M^B$ forgetting the $B$-module structure is Frobenius with left and right adjoint $-\tildetens{} B:\M^B\to \M^B_B$.
\end{proposition}

\begin{proof}
Since $B$ is in particular a Frobenius algebra, we know that there exists a bijection
\begin{equation}\label{eq:FrobeniusAdj}
\begin{gathered}
\xymatrix @R=0pt{
\phi:\Hom{}{}{}{}{U(M)}{V} \ar@{<->}[r] & \Hom{}{}{B}{}{M}{V\otimes B} \\
f \ar@{|->}[r] & \left[ n \mapsto f(ne^1)\otimes e^2\right] \\
\left[ n \mapsto (V\otimes \psi)(g(n)) \right] & g \ar@{|->}[l]
}
\end{gathered}
\end{equation}
for every $M\in\M_B$ and $V\in \M$. 
Consider the unit $\eta_M:M\to U(M)\otimes B,\,m\mapsto me^1\otimes e^2$, and the counit $\epsilon_V:U(V\otimes B)\to V,\,v\otimes b\mapsto v\psi(b),$ of \eqref{eq:FrobeniusAdj}. In light of Lemma \ref{lemma:CoinvCasimir}, we have
\begin{gather*}
\delta_{U_B(M)\tildetens{} B}\left(\eta_M(m)\right) = \delta_{U_B(M)\tildetens{} B}\left(me^1\otimes e^2\right) = (me^1)_0\otimes e^2_1 \otimes (me^1)_1e^2_2 \\
= m_0e^1_1\otimes e^2_1 \otimes m_1e^1_2e^2_2 = m_0e^1\otimes e^2 \otimes m_1 = \eta_M(m_0)\otimes m_1,
\end{gather*}
so that $\eta_M\in \Hom{}{}{B}{B}{M}{U_B(M)\tildetens{} B}$, and since $\psi\in\int^rB^*$ we have
\[
(V\otimes \psi\otimes B)\left(\delta_{V\otimes B}(v\otimes b)\right) = v_0\psi(b_1)\otimes v_1b_2 = v_0\otimes v_1\psi(b_1)b_2 = v_0\otimes v_1\psi(b) = \delta_V\left(\epsilon_V(v\otimes b)\right),
\]
so that $\epsilon_V\in \Hom{}{}{}{B}{U_B(V\tildetens{} B)}{V}$. Therefore \eqref{eq:FrobeniusAdj} induces a bijection
\[
\Hom{}{}{}{B}{U_B(M)}{P} \cong \Hom{}{}{B}{B}{M}{P\tildetens{} B}
\]
for $M\in\M^B_B$ and $P\in\M^B$. Concerning the last claim, since $B$ is a monoid in the monoidal category of right $B$-comodules it is well-known that $-\tildetens{} B:\M^B\to \M^B_B$ is always left adjoint to $U_B:\M^B_B\to \M^B$.
\end{proof}

\begin{lemma}\label{lemma:CoinvFrobenius}
For every $M\in\M^B_B$ and every $V\in\M$ the assignment
\[
\phi_{M,V} : \Hom{}{}{}{B}{U_B(M)}{V^u} \to \Hom{}{}{}{}{\coinv{M}{B}}{V}: \quad f\mapsto \left[m\mapsto f\left(me^1\varepsilon(e^2)\right)\right]
\]
provides a bijection $\Hom{}{}{}{B}{U_B(M)}{V^u}\cong\Hom{}{}{}{}{\coinv{M}{B}}{V}$, natural in $M$ and $V$, with explicit inverse
\[
\varphi_{M,V} : \Hom{}{}{}{}{\coinv{M}{B}}{V} \to \Hom{}{}{}{B}{U_B(M)}{V^u} : \quad g \mapsto \left[ m \mapsto g\left(m_0S(m_1)\right)\psi(m_2)\right].
\]
\end{lemma}

\begin{proof}
Set $\phi:=\phi_{M,V}$ and $\varphi:=\varphi_{M,V}$ for the sake of brevity and recall that for a finitely generated and projective $\K$-bialgebra $B$, $\psi$ is a right integral on $B$ if and only if $(\psi\otimes B)\circ\Delta = \psi\otimes u$. Recall also that $e^1\varepsilon(e^2)=t$ is the left norm in $B$ with respect to $\psi$, whence we can rewrite $\phi(f)=t\cdot f$ and $\varphi(g) = (g\otimes \psi)\circ \theta^{-1}_M$ where $\theta$ is the one of \eqref{eq:unitscounits}. The first assignment is clearly well-defined. For what concerns the second one, the following computation
\[
(g\otimes \psi\otimes B) \circ (\theta_M^{-1}\otimes B)\circ \delta_M = (g\otimes \psi\otimes B) \circ (\coinv{M}{B}\otimes \Delta) \circ \theta_M^{-1} = (V\otimes u)\circ (g\otimes \psi) \circ \theta_M^{-1}
\]
proves that $\varphi(g)$ is colinear, so that $\varphi$ is well-defined. Let us prove that $\phi$ and $\varphi$ are each other inverses.
On the one hand, for all $m\in\coinv{M}{B}$ and $g\in \Hom{}{}{}{}{\coinv{M}{B}}{V}$ we have
\[
\phi(\varphi(g))(m) = \varphi(g)(mt) = g\left(m_0t_1S(m_1t_2)\right)\psi(m_2t_3) = g\left(mt_1S(t_2)\right)\psi(t_3) = g(m),
\]
which proves that $\phi\circ\varphi$ is the identity. On the other hand, recall from Lemma \ref{lemma:CoinvCasimir} that we have $e=e^1\otimes e^2\in\coinv{(B^\bullet \otimes B^\bullet)}{B}$, so that $t\otimes 1 = e^1_1\otimes e^1_2e^2$, and that $at=\varepsilon(a)t$ for all $a\in B$. For every $m\in M$ and $f\in \Hom{}{}{}{B}{U_B(M)}{V^u}$ compute
\begin{gather*}
\varphi(\phi(f))(m) = \phi(f)\left(m_0S(m_1)\right)\psi(m_2) = f\left(m_0S(m_1)t\right)\psi(m_2) = f(m_0t)\psi(m_1) \\
 = f\left(m_0e^1_1\right)\psi(m_1e^1_2e^2) \stackrel{(*)}{=} f\left(me^1\right)\psi(e^2) = f(m),
\end{gather*}
where $(*)$ follows from colinearity of $f$:
\begin{gather*}
f(m_0e^1_1)\psi(m_1e_2^1e^2) = \left(V\otimes (\psi\circ m_B)\right)\left( f\left((me^1)_0\right) \otimes (me^1)_1 \otimes e^2\right) \\
= \left(V\otimes (\psi\circ m_B)\right)\left( f\left(me^1\right) \otimes 1 \otimes e^2\right) = f\left(me^1\right)\psi(e^2).
\end{gather*}
Therefore $\varphi\circ\phi$ is the identity as well. We are left to check that $\phi_{M,V}$ is natural. To this aim, consider $\alpha:M'\to M$ in $\M_B^B$ and $\beta: V \to V'$ in $\M$. Then for every $f\in \Hom{}{}{}{B}{U_B(M)}{V^u}$ and $m\in M'$ we have
\[
\phi_{M',V'}(\beta\circ f \circ \alpha)(m) = \beta\left(f\left(\alpha\left(mt\right)\right)\right) = \beta\left(f\left(\alpha\left(m\right)t\right)\right) = \left(\beta\circ \phi_{M,V}(f)\circ \alpha\right)(m). \qedhere
\]
\end{proof}

Consider now the adjoint triple
\[
U^B \ \dashv \ -\cmptens{} B \ \dashv \ \Hom{}{}{B}{B}{B\cmptens{}B}{-}
\]
between $\rhopf{B}$ and $\Rmod{B}$, with units and counits
\begin{gather*}
\eta_M:M\to U^B(M)\cmptens{} B, \quad m\mapsto m_0\otimes m_1, \qquad \epsilon_N:U^B(N\cmptens{}B) \to N, \quad n\otimes b\mapsto n\varepsilon(b), \\
\gamma_N: N \to \Hom{}{}{B}{B}{B\cmptens{}B}{ N\cmptens{}B}, \quad n\mapsto [a\otimes b\mapsto n\cdot a\otimes b], \\
\theta_M: \Hom{}{}{B}{B}{B\cmptens{}B}{ M}\cmptens{}B \to M, \quad f\otimes a\mapsto f(1\otimes a).
\end{gather*}

\begin{proposition}\label{prop:UBFrob}
The assignment
\[
\Gamma: \Hom{}{}{B}{B}{B\cmptens{}B}{M} \to U^B(M): \quad f\mapsto f(e^1e_1^2\otimes e^2_2)=f(e^1\otimes 1)\cdot e^2
\]
is an isomorphism of right $B$-modules, natural in $M\in \M^B_B$, with inverse given by
\[
\Lambda: U^B(M) \to \Hom{}{}{B}{B}{B\cmptens{}B}{M}: \quad m\mapsto \big[a\otimes b\mapsto m_0\cdot S(m_1)\psi(m_2aS(b_1))b_2\big].
\]
In particular, the functor $-\cmptens{} B:\M_B\to\M_B^B$ is Frobenius with left and right adjoint the functor $U^B:\M^B_B\to\M_B$. Explicitly,
\[
\xymatrix @R=0pt{
\Hom{}{}{B}{B}{N\cmptens{}B}{M} \ar@{<->}[r] & \Hom{}{}{B}{}{N}{U^B(M)} \\
f \ar@{|->}[r] & \big[ n \mapsto f\left(n\cdot e^1\otimes 1\right)\cdot e^2\big] \\
\big[ n\otimes b \mapsto g(n)_0\cdot S\left(g(n)_1\right)\psi\left(g(n)_2S\left(b_1\right)\right)b_2\big] & g \ar@{|->}[l]
}
\]
\end{proposition}

\begin{proof}
We may compute directly
\begin{align*}
 & \Lambda(\Gamma(f))(a\otimes b) = \Lambda(f(e^1e_1^2\otimes e^2_2))(a\otimes b) \\
 & = f(e^1e_1^2\otimes e^2_2)_0\cdot S\left(f(e^1e_1^2\otimes e^2_2)_1\right)\psi(f(e^1e_1^2\otimes e^2_2)_2aS(b_1))b_2 \\
 & = f(e^1e_1^2\otimes e^2_2)\cdot S(e^2_3)\psi(e^2_4aS(b_1))b_2 = f(e^1\otimes 1)e_1^2S(e^2_2)\psi(e^2_3aS(b_1))b_2 \\
 & = f(e^1\otimes 1)\cdot \psi(e^2aS(b_1))b_2 = f(aS(b_1)e^1\otimes 1)\psi(e^2)b_2 \\
 & = f(aS(b_1)\otimes 1)\cdot b_2 = f(aS(b_1)b_2\otimes b_3) = f(a\otimes b)
\end{align*}
for all $f\in \Hom{}{}{B}{B}{B\cmptens{}B}{M}$ and $a,b\in B$ and
\begin{align*}
\Gamma(\Lambda(m)) & = \Lambda(m)(e^1\otimes 1)\cdot e^2  =m_0\cdot S(m_1)\psi\left(m_2e^1\right)e^2 = m_0\cdot S(m_1)\psi(e^1)e^2m_2 = m
\end{align*}
for all $m\in M$, so that both $\Gamma\circ\Lambda$ and $\Lambda \circ \Gamma$ are the identity morphism. For what concerns the explicit bijection giving that $-\cmptens{}B$ is also left adjoint to $U^B$, it can be easily deduced by using $\gamma$, $\theta$, $\Gamma$ and $\Lambda$.
\end{proof}

\begin{lemma}\label{lemma:ClFrobenius}
For every $M\in\rhopf{B}$ and $V\in\M$ the assignment
\[
\phi_{M,V}':\Hom{}{}{B}{}{V_\varepsilon}{U^B(M)} \to \Hom{}{}{}{}{V}{\cl{M}^B}: \quad f \mapsto \left[v\mapsto \cl{f(v)_0}\psi\left(f(v)_1\right)\right]
\]
provides a bijection $\Hom{}{}{B}{}{V_\varepsilon}{U^B(M)} \cong \Hom{}{}{}{}{V}{\cl{M}^B}$, natural in $M\in\rhopf{B}$ and $V\in\M$, with explicit inverse
\[
\varphi'_{M,V} : \Hom{}{}{}{}{V}{\cl{M}^B} \to \Hom{}{}{B}{}{V_\varepsilon}{U^B(M)}: \quad g\mapsto \left[g(v)'_0\cdot S\left(g(v)'_1\right)\varepsilon(e^1)e^2\right],
\]
where $g(v)'\in M$ is any element such that $g(v)=\cl{g(v)'}$.
\end{lemma}

\begin{proof}
We leave to the reader to check that $\varphi'_{M,V}$ is well-defined. Notice that for every $f\in \Hom{}{}{B}{}{V_\varepsilon}{U^B(M)}$ we have $f(v)\cdot b=f(v)\varepsilon(b)$ and compute
\begin{gather*}
\varphi'_{M,V}\left(\phi'_{M,V}(f)\right)(v) = f(v)_0\cdot S(f(v)_1)\psi(f(v)_2)\varepsilon(e^1)e^2 = f(v)_0\cdot e^1_1S(f(v)_1e^1_2)\psi(f(v)_2e^1_3)e^2 \\
 = f(v)_0\cdot S(f(v)_1)\psi(f(v)_2e^1)e^2 = f(v)_0\cdot S(f(v)_1)\psi(e^1)e^2f(v)_2 \stackrel{\eqref{eq:FrobRel}}{=} f(v)
\end{gather*}
for every $v\in V$, whence $\varphi'_{M,V}\circ\phi'_{M,V}$ is the identity. The other way around, for every $g\in \Hom{}{}{}{}{V}{\cl{M}^B}$ we have
\begin{gather*}
\phi'_{M,V}\left(\varphi'_{M,V}(g)\right)(v) = \cl{g(v)'_0\cdot S\left(g(v)'_2\right)_1\varepsilon(e^1)e^2_1\psi\left(g(v)'_1S\left(g(v)'_2\right)_2e^2_2\right)} \\
 = \cl{g(v)'_0\cdot S\left(g(v)'_3\right)\varepsilon(e^1)e^2_1\psi\left(g(v)'_1S\left(g(v)'_2\right)e^2_2\right)}  = \cl{g(v)'_0\cdot S\left(g(v)'_1\right)\varepsilon(e^1)e^2_1\psi\left(e^2_2\right)} \\
 \stackrel{(*)}{=} \cl{g(v)'_0}\varepsilon\left(S\left(g(v)'_1\right)\right)\varepsilon(e^1)\varepsilon\left(e^2_1\right)\psi\left(e^2_2\right) = \cl{g(v)'}\varepsilon\left(e^1\psi\left(e^2\right)\right) \stackrel{\eqref{eq:FrobRel}}{=} g(v)
\end{gather*}
for all $v\in V$, where $(*)$ follows from the fact that $\cl{mb}=\cl{m}\varepsilon(b)$. Therefore, $\phi'_{M,V}\circ\varphi'_{M,V}$ is the identity as well.
\end{proof}

In conclusion, we have the following result.

\begin{theorem}\label{thm:summingup}
The following are equivalent for a finitely generated and projective $\K$-bialgebra $H$.
\begin{enumerate}[label=(\arabic*), ref=\emph{(\arabic*)}, leftmargin=0.8cm, labelsep=0.25cm]
\item\label{item:summingup1} The functor $-\bartens{}H:\M\to\rhopf{H}$ is Frobenius and $\int_r\rdual{H}\cong \K$.
\item\label{item:summingup2} $H$ is a Hopf algebra with $\int_r\rdual{H}\cong \K$.
\item\label{item:summingup3} $H$ is an FH-algebra.
\item\label{item:summingup4} The functor $-\tildetens{}H:\Rcomod{H}\hspace{-2pt}\to\rhopf{H}$ is Frobenius and $\mathsf{Hom}^{H}(U_H(M),V^u) \hspace{-1pt} \cong \hspace{-1pt} \mathsf{Hom}(\coinv{M}{H},V)$, naturally in $M\in \rhopf{H}, V\in\M$.
\item\label{item:summingup5} The functor $-\bartens{}H:\M\to\rhopf{H}$ is Frobenius and $\int_r{H}\cong \K$.
\item\label{item:summingup6} $\rdual{H}$ is a Hopf algebra with $\int_r{H^{**}}\cong \K$.
\item\label{item:summingup7} $H^*$ is an FH-algebra.
\item\label{item:summingup8} The functor $-\cmptens{}H:\Rmod{H}\to\rhopf{H}$ is Frobenius and $\mathsf{Hom}_H(V_\varepsilon,U^H(M)) \cong \mathsf{Hom}(V,\cl{M}^H)$, naturally in $M\in \rhopf{H}, V\in\M$.
\end{enumerate}
\end{theorem}

\begin{proof}
Let us show firstly that $\ref{item:summingup1}\Leftrightarrow\ref{item:summingup2}\Leftrightarrow\ref{item:summingup3}\Leftrightarrow\ref{item:summingup4}$. The implication from \ref{item:summingup1} to \ref{item:summingup2} follows from Theorem \ref{thm:sigmainv2} and \cite[Proposition 5]{GreenNicholsTaft}. The one from \ref{item:summingup2} to \ref{item:summingup3} is the right-handed analogue of Pareigis' \cite{Pareigis}. The fact that \ref{item:summingup3} implies \ref{item:summingup4} is the content of Proposition \ref{prop:ForgetModFrobenius} and Lemma \ref{lemma:CoinvFrobenius}. Finally, $\ref{item:summingup4}\Rightarrow\ref{item:summingup1}$ follows from Proposition \ref{prop:ManyAdj} \ref{item:ManyAdj1} and the observation that $\int_r\rdual{H}\cong \Hom{}{}{}{H}{U_H(H)}{\K^u}\cong \Hom{}{}{}{}{\coinv{H}{H}}{\K}\cong\K$.

Secondly, let us prove that $\ref{item:summingup3}\Leftrightarrow\ref{item:summingup5}\Leftrightarrow\ref{item:summingup6}\Leftrightarrow\ref{item:summingup7}\Leftrightarrow\ref{item:summingup8}$. If \ref{item:summingup5} holds then $H$ is a Hopf algebra and $\int_rH\cong \int_rH^{**}$, whence we have \ref{item:summingup6}. The implication from \ref{item:summingup6} to \ref{item:summingup7} is again Pareigis' result applied to $\rdual{H}$.  The one from \ref{item:summingup7} to \ref{item:summingup3} is the content of \cite[Proposition 4.3]{KadisonStolin}. The fact that \ref{item:summingup3} implies \ref{item:summingup8} follows from Proposition \ref{prop:UBFrob} and Lemma \ref{lemma:ClFrobenius} and, lastly, the implication from \ref{item:summingup8} to \ref{item:summingup5} is Proposition \ref{prop:ManyAdj} \ref{item:ManyAdj2} and the observation that $\int_r{H}\cong \Hom{}{}{H}{}{\K_\varepsilon}{U^H(H)}\cong \Hom{}{}{}{}{\K}{\cl{H}^H}\cong\K$.
\end{proof}


\end{document}